\documentclass[reqno]{amsart}
\usepackage[utf8]{inputenc} 
\usepackage{amssymb}
\usepackage[initials]{amsrefs} 
\usepackage{array}    
\usepackage{graphicx} 
\usepackage[counterclockwise]{rotating} 
\usepackage{verbatim} 
\usepackage{tikz}
\usepackage{tikz-cd}  

\usepackage[font=small]{caption}

\usepackage{hyperref} 

\numberwithin{subsection}{section}
\numberwithin{equation}{section}
\numberwithin{table}   {section}
\numberwithin{figure}  {section}

\theoremstyle{plain}
 \newtheorem {theorem}    {Theorem}[section]
 
 \newtheorem {lemma}      [theorem]{Lemma}
 \newtheorem {corollary}  [theorem]{Corollary}
\theoremstyle{definition}
 \newtheorem {definition} [theorem]{Definition}

\theoremstyle{remark}
 \newtheorem {remark}     [theorem]{Remark}
\newcommand{\missarg}{\,\cdot\,}

\newcommand{\mathC}{\mathbb{C}}
\newcommand{\mathN}{\mathbb{N}}
\newcommand{\mathZ}{\mathbb{Z}}
\newcommand{\mathR}{\mathbb{R}}

\newcommand{\calE} {\mathcal{E}}	
\newcommand{\calH} {\mathcal{H}}
\newcommand{\calF} {\mathcal{F}}    
\newcommand{\calZ} {\mathcal{Z}}	
\newcommand{\calR} {\mathcal{R}}	
\newcommand{\calV} {\mathcal{V}}	
\newcommand{\calN} {\mathcal{N}} 
\newcommand{\Poiss}{K}      				  

\newcommand{\pcrit}{p_{\text{\rm crit}}}

\newcommand{\barra}{\,|\, }

\newcommand{\AutT}{\mathcal G}

\DeclareMathOperator{\Aut} {Aut}
\DeclareMathOperator{\Int}{Int}
\DeclareMathOperator{\dist}{dist}

\DeclareMathOperator{\spectrum}{sp}

\def\psiphi{\phi}

\begin{document}

\title
 [Positive definite functions on semi-homogeneous trees]
 {Positive definite functions\\on semi-homogeneous trees\\and spherical representations}
\author
{Massimo~A. Picardello}
\address
{Department of Mathematics\\
 Tor Vergata University of Rome\\
 Via della Ricerca Scientifica, 00133 Rome, Italy}
 
\thanks{
Partially supported by MIUR Excellence Departments Project awarded to
the Department of Mathematics, Tor Vergata University of Rome,  MatMod\@TOV
}
\email{picard@mat.uniroma2.it}

\subjclass
 {Primary 44A12;
Secondary 05C05,
          43A85}

\keywords
{Homogeneous and semi-homogeneous trees, 
 Laplace operators, 
spherical functions, positive-definite functions}

\begin{abstract}
{We consider the group $\AutT$ of isometries of a semi-homogeneous tree $T=T_{q_+,q_-}$ with valencies $q_+ +1$ and $q_- +1$ and its two orbits $V_+$, $V_-$ respectively. We make use of the action of $\AutT$ to equip the spaces of finitely supported radial functions on each of $V_\pm$ with convolution products, hence with a notion of positive definite functions. The  $\ell^1$-functions radial around a root vertex $v_0\in V_+$ form an abelian convolution algebra.
We study its multiplicative functionals, called spherical functions, given by eigenfunctions of the nearest-neighbor isotropic transition operator (the Laplace operator on $T$,
and determine which of them are positive definite. Each positive definite function gives rise to a unitary representation of $\AutT$; in this way, we produce a series of unitary spherical representations. For $q_+<q_-$, the representation whose spherical function has eigenvalue 0 is square-integrable.}
\end{abstract}

\maketitle

%

%
\section{Introduction}
Let $T$ be a homogeneous tree, that is, a graph without loops and the same number $q+1$ of neighbors for each vertex; $q$ is called the homogeneity degree.
The isotropic nearest-neighbor transition operator $\mu_1$  its set $V$ of vertices ( called  Laplace operator) has been studied since a long time (see \cite{Figa-Talamanca&Picardello} for references), and is related to the Poisson boundary $\Omega$ of $T$ (a topological boundary that yields a compactification of $V$) and its Poisson kernel $\Poiss$. The  setup of a  homogeneous tree has close ties to harmonic analysis, because functions on $V$ are endowed with a convolution product induced by its group of isometries $\AutT=\Aut T$ (or a simply transitive subgroup thereof), and  $\mu_1$ is a convolution operator. Appropriate estimates for the norms of convolution operators \cite{Haagerup} can then be used to determine the spectrum of $\mu_1$ on $\ell^p(V)$ for $1\leqslant p <\infty$ \cite{Figa-Talamanca&Picardello}*{Chapter 3, Theorem 3.3}. Choose a root vertex $v_0$ and denote by $\AutT_{v_0}$ its stability subgroup in $\AutT$. Then $(\AutT,\AutT_{v_0})$ is a Gelfand pair, that is, with the convolution product induced by $\AutT$, the algebra of two-sided $\AutT_{v_0}$-invariant function on $\AutT$ (or, in an equivalent realization, the algebra of  functions on $V$ radial around $v_0$) is commutative; that is, $(\AutT,\AutT_{v_0})$ is a Gelfand pair.
 Its spherical functions are the eigenfunctions $\phi(\missarg,v_0\barra \gamma)$ of $\mu_1$  radial around $v_0$ and normalized to $1$ at $v_0$: at the eigenvalue $\gamma$, the spherical function, that we denote by
 $\phi(\missarg,v_0\barra \gamma)$,
 was computed in \cite{Figa-Talamanca&Picardello-JFA}.
 
 It was shown in \cite{Figa-Talamanca&Picardello} that the spectrum of $\mu_1$ on $\ell^p(V)$ is the closure of the set $\{\gamma\in\mathC\colon \phi(\missarg,v_0\barra \gamma)\in\ell^r(V) \text{ for every }r>p\}$. Moreover, since $\mu_1$ is invariant under $\AutT$, its eigenspaces give rise to spaces of representations of $\AutT$,  called \emph{spherical representations}. The  representation at the eigenvalue $\gamma$ is unitary if $\gamma$ belongs to the $\ell^2$-spectrum of $\mu_1$, and unitarizable if $\phi(\missarg,v_0\barra \gamma)$ is positive definite.

A recent article \cite{Casadio-Tarabusi&Picardello-semihomogeneous_spectra}  considers trees that are \emph{semi-homogeneous}, i.e., with two  alternating homogeneity degrees $q_+$ and $q_-$.  The semi-homogeneous Laplace operator $\mu_1$ on functions on $V$ is again as the average operator on neighbors, and gives rise to a nearest-neighbor transition operator.  Its spherical functions are obtained as generalized Poisson transforms, by means of an explicit realization of the Poisson kernel for each eigenvalue $\gamma$, derived by an explicit computation of the first visit probability associated to the random walk generated by $\mu_1$.

On the other hand, the semi-homogeneous and the homogeneous settings are considerably different. Indeed, the group $\AutT$ has two orbits $V_+$ and $V_-$, hence it is not transitive and does not give rise to a convolution product. For the same reason, the natural definition of positive definite function does not make sense. 
As a consequence, we do not have the convolution  estimates of \cite{Haagerup} that have been used in the literature to compute the spectrum of $\mu_1$ (see \cite{Figa-Talamanca&Picardello}*{Chapter 3}), and it is much harder to build an analytic series of spherical representations of $\AutT$ by means of spherical functions.

Indeed, the basic problem raised by the existence of two orbits under $\AutT$  is the lack of a convolution product, hence of a suitable convolution algebra $\calR$ of radial functions (see also \cite{Casadio-Tarabusi&Picardello-algebras_generated_by_Laplacians}). In the homogeneous setting, the spectrum of $\mu_1$ is obtained from the multiplicative functionals on the  Banach algebra generated by $\calR$ in the $\ell^1$-norm. Here we need a new approach. The operator $\mu_1$, being a nearest-neighbor transition operator, induces jumps only between vertices of different homogeneities. An idea developed in \cite{Casadio-Tarabusi&Picardello-semihomogeneous_spectra} with the goal of computing the $\ell^p$-spectrum of $\mu_1$ focuses onto
its square $\mu_1^2$, that preserves the homogeneity, hence  acts on the two orbits separately, and is transitive on each of them. On the other hand, $\mu_1^2$ is related to the step-2 isotropic transition operator $\mu_2$, indeed, up to a scale factor, it coincided with $\mu_2$ plus a positive multiple of the identity.

On the other hand, the notion of adjacency induced by $\mu_2$ on $V_+$, that corresponds to distance 2 in $V$, transforms $V_+$ into the polygonal graphs $\Gamma_+$, studied in \cite{Iozzi&Picardello-Springer}, that consists of infinitely many complete polygons with $q_-+1$ vertices attached in a tree-like fashion, with $q_+ +1$ polygons joining at each vertex; a similar description holds for the polygonal graph $\Gamma_-$ that corresponds to the action on $V_-$. The action is transitive and gives rise on, say, $V_+$ to a nice convolution algebra of functions radial in this graph around a fixed reference vertex $v_0\in V_+$. For every $p$, the $\ell^p$-spectrum of $\mu_2$ on $V_\pm$ is known
\cite{Iozzi&Picardello-Springer}. By making use of this,  the $\ell^p$-spectrum of $\mu_1$  has been computed in \cite{Casadio-Tarabusi&Picardello-semihomogeneous_spectra}

An unusual fact occurs in the semi-homogeneous setting  \cite{Casadio-Tarabusi&Picardello-semihomogeneous_spectra}: there is  a spherical function that belongs to $\ell^2(V)$, and also to $\ell^p(V)$ for some $p<2$. Indeed, the spherical function at the eigenvalue 0 belongs to $\ell^p(V)$ for every $p>1+
{\ln q_+}/{\ln q_-}$, that is smaller than 2 if and only if $q_+<q_-$. Every other bounded spherical function $\phi(\missarg,v_0\barra \gamma)$ belongs to $\ell^p(V)$ for some $p>2$ depending on $\gamma$.

The lack of transitivity of $\AutT$ gives rise to some interesting questions.
Since $\mu_1$ is $\AutT$-invariant, i.e., it commutes with $\AutT$,  its eigenspaces are also invariant under $\AutT$, hence they give rise to representations of this group, namely,  the spherical representations. The representation space at the eigenvalue $\gamma$ is the span of all translates of $\phi(\missarg,v_0\barra \gamma)$. In the homogeneous setting, this representation is unitary or unitarizable if and only if $\phi(\missarg,v_0\barra \gamma)$ is positive definite (thereby providing the necessary Hilbert norm via the GNS construction) \cites{Figa-Talamanca&Picardello-JFA, Figa-Talamanca&Picardello}. In the semi-homogeneous setting, the notion of positive definite makes sense for functions defined on the group $\AutT$, but not for functions on $V$, that cannot be lifted to  $\AutT$ since the group action is not transitive. Even though  there are only two orbits, it is not known how to equip the space of functions on $V$ with a suitable definition of positive definite function, or at least a Hilbert norm on the linear span of all translates of $\phi(\missarg,v_0\barra \gamma)$. Therefore it is not clear which spherical representations are unitary.

For the same reason, $\AutT$ does not induce a convolution product on functions on $V$. On the other hand, both $V_+$ and on $V_-$ are homogeneous spaces for $\AutT$, indeed $V_\pm=\AutT/\AutT_{v_\pm}$ for any choice of reference vertices $v_\pm\in V_\pm$, one for each homogeneity.
Between two functions defined on each of these homogeneous spaces $\AutT$ induces a convolution product, but only if one of the two functions is bi-$\Aut_{v_\pm}$-invariant, as follows.

For simplicity, let us restrict attention to functions on $V_+$ and denote again by $v_0$ the reference vertex in $V_+$. If $v\in V$ and $\lambda\in\AutT$ is such that $v=\lambda v_0$, the convolution product induced by $\AutT$ on its orbit $\AutT/\AutT_{v_0}\approx V_+$ should be defined as $f*g(v)= \int_{\AutT} f( \tau^{-1}\lambda v_0)\,g(\tau v_0) \,d\tau$, where the measure $d\tau$ is the Haar measure of the (unimodular) group $\AutT$. On the other hand, the result is invariant on right cosets only if $f$ is two-sided invariant, that is, radial around $v_0$ (see Section \ref{Sec:multiplicative_functionals} for more details). In particular, we obtain a radial convolution algebra on $V_+$ (and another on $V_-$). This actually leads to two different convolution products, one for each orbit (see \cite{CCKP}). It is not clear how to define the convolution of functions not supported on a single orbit in such a way that it coincides with the usual definition when the tree is homogeneous. Moreover, an analogous of Haagerup's convolution estimate \cite{Haagerup}, that is the main tool for computing the spectra of the Laplacian in the well-established approach of \cite{Figa-Talamanca&Picardello}*{Chapter 3}, is complicated in this setting \cite{Iozzi&Picardello}.

The spherical representation $\pi_\gamma$ of a group $G$ acting simply transitively on a homogeneous tree is known to be irreducible for $\gamma\in\spectrum_{\ell^1}(\mu_1)$, by \cite{Figa-Talamanca&Picardello}*{Chapter 5} and \cite{Figa-Talamanca&Nebbia}. The argument, later extended  to non-isotropic nearest-neighbor transition operators  \cite{Figa-Talamanca&Steger}, proceeds by proving that the projector onto a cyclic vector is the weak limit, as $n\to\infty$, of the average of $\pi_\gamma(\tau)$ over all the elements $\tau\in \AutT$ that move $v_0$ to vertices of length $n$. Here again, the lack of transitivity makes it unhandy to reproduce the same argument. Instead of an irreducible representation, we obtain a representations reducible as the direct sum of two components, one for each orbit.

In this short note, we study unitary representation of $\AutT$ by means of its action on $V_+\approx \AutT/\AutT_{v_0}$ (or the analogous action on $V_-$) and the convolution induced by this action on the radial space $\ell^1_\#(V_+)$, that turns out to be the abelian convolution algebra generated by the Laplacian. The group $\AutT$ can be factorized as $\AutT=\calF \AutT_{v_0}$, where $\calF$ is any discrete subgroup acting simply transitively on $V_+$. The convolution product is much simpler when considered on $\calF$, whose Cayley graph is the polygonal graph $\Gamma_+$ associated to $V_+$. A representation theory for $\calF$ follows in the same way with more readable statements and proof, but we prefer to work on the larger group $\AutT$ that is naturally associated to the whole tree; observe that any isometry of $T$ restricts to isometries of $V_\pm$, and conversely, every isometry of $V_+$ (or $V_-$) extends uniquely to an isometry of $T$. Besides, the choice of $\calF$ is not unique; for instance, the choice of $\calF$ in \cite{Iozzi&Picardello-Springer} is the free product of $q_+ +1$ copies of $\mathZ_{q_-+1}$; another choice is $\mathZ_{q_+ +1}*\mathZ_{q_- +1}$, that is the discrete group naturally associated with the dual graph of $\Gamma_+$.  For an easier understanding, the reader is invited to rephrase our proofs for $\calF$, and compare our results with those of  \cite{Iozzi&Picardello-Springer}, that yield the $\ell^p$-spectrum of the Laplacian on $\Gamma_+$. For the $\ell^p$-spectrum of the Laplacian on( a semi-homogeneous tree, see \cite{Casadio-Tarabusi&Picardello-semihomogeneous_spectra};
here we do not study spectra, only positive definite functions and group representations.

We make use of the action of $\AutT$ on $V_+$ and its natural involution to introduce positive definite functions on $V_+$. Moreover,
 we study the multiplicative functionals on the algebra $\ell^1_\#(V_+)$, that correspond to the normalized bounded radial eigenfunctions of the Laplacian \emph{(spherical functions)}, prove that the spherical functions are positive definite if and only if they are real-valued and bounded, and conclude that the spherical functions that correspond to real eigenvalues in the $\ell^1$-spectrum of the Laplacian give rise, via the GNS construction, to a family of unitary representations of $\AutT$ that are irreducible on $\ell^2(V_\pm)$, whereas they decompose as direct sum of two irreducible representations on $\ell^2(V)=\ell^2(V_+)\oplus \ell^2(V_-)$. For $q_+<q_-$, the spherical function at the eigenvalue 0 belongs to $\ell^2(V_+)$ and is positive definite, and the representation is square-integrable, that is, it is a unitary subrepresentation of the regular representation of $\AutT$. It is worth noting that, for $q_+ < q_-$, a discrete subrepresentation of the regular representation of the free product of $q_+ +1$ copies of $\mathZ_{q_- +1}$ (the discrete group acting simply transitively on $\Gamma_+$ introduced above) was found in \cites{Kuhn&Soardi, Faraut&Picardello}.

The author acknowledges many enlightning conversations with Enrico Casadio-Tarabusi. An alternative approach to the results of this article will appear in a forthcoming joint paper.

\section{Semi-homogeneous trees and the Laplace operator}
\label{Sec:semi-homogeneous_horospheres_and_circles}

A \emph{tree} $T$ is a connected, countably infinite, locally finite graph without loops. The nodes of $T$ are called \emph{vertices}; the set of all vertices is denoted by $V$. Two distinct vertices $v,v'$ are \emph{adjacent}, or \emph{neighbors}, if they belong to the same edge: we shall write $v\sim v'$.
Let us fix a reference vertex $v_0$. The number of neighbors of $v$ is $q_v+1$, where $q_v$ is the  \emph{homogeneity degree}, that is, the number of outward neighbors of $v$ with respect to $v_0$ if $v\neq v_0$.
Let us fix a \emph{parity}, that is, an alternating function $\epsilon\colon V\to \{\pm 1\}$. The level sets of $\epsilon$ are denoted by $V_+$ and $V_-$. 

A \emph{semi-homogeneous} tree $T=T_{q_+,\,q_-}$ has two alternating homogeneity degrees ${q_+}$ on $V_+$ and ${q_-}$ on $V_-$. If ${q_+}={q_-}$ then the tree is \emph{homogeneous}. We  assume $q_+, q_->1$ and choose  $v_0\in V_+$, that is,
\begin{equation}\label{eq:q_{v_0}=q_+}
    q_{v_0}=q_+.
\end{equation}
%

Since it has no loops, a homogeneous tree is the Cayley graph of the free product $\mathfrak F$ 
of $q+1$ copies of the two-element group $\mathZ_2$. Indeed, this free product embeds in the group $\AutT$ of \textit{automorphisms} of $T$, the invertible self-maps of $V$ that preserve adjacency: the generators of the factors are  automorphisms that reverse the edges that contain $v_0$ (see, for instance, \cite{Figa-Talamanca&Picardello}*{Chapter 3, Section V}), and we have the factorization $\AutT=\mathfrak F\, \AutT_{v_0}$. In particular, $\mathfrak F$ acts simply transitively on $T$ and induces a convolution product on functions on $V$. 
Therefore $\AutT$ is transitive on $V$ if $T$ is homogeneous.  On the space of functions on $\AutT$ that are two-sided-invariant  under the stability subgroup $\AutT_{v_0}\subset \AutT$ of $v_0$, the convolution operation induced by the group $\AutT$ is a lifting of the convolution induced by 
$\mathfrak F$
(for homogeneous trees, see \cite{Casadio-Tarabusi&Picardello-Radon_on_trees_and_counterparts}*{Appendix} and
\cite{Casadio-Tarabusi&Gindikin&Picardello-Book}*{Subsection 3.1.1}
).
 On the other hand, if ${q_+}\neq {q_-}$, then there are two orbits of $\AutT$ on $V$, namely $V_{+}$ and  $V_{-}$.

At each vertex $v$, the Laplace operator applied to a function $f$ on $V$ yields the average of the values of $f$ at the neighbors of $v$. The number of neighbors depends on the parity of the vertex. That is,
\begin{equation}\label{eq:semihomogenous_Laplacian}
\mu_1 f(v) =
\frac1{q_{\epsilon(v)}+1}\; \sum_{w\sim v} f(w).
\end{equation}
%

\begin{remark}\label{rem:recurrence_relation_of_the_semi-homogeneous_Laplacian}
Following \cite{Casadio-Tarabusi&Picardello-semihomogeneous_spectra}, let us consider the space $\mathfrak R$ of summation operators with kernel, acting on $\ell^1(V)$, that is,  of the type $R f(v)=\sum_{w\in V} r(v,w) f(w)$, whose kernel $r(v,w)$ is of  finite range (that is, it vanishes if $\dist(v,w)>N$ for some $N\in\mathN$ depending on $r$). Then $\mu_1$ belongs to $\mathfrak R$. More generally, for $n\in\mathN$  we set
\begin{equation}\label{eq:M_n}
\mu_n f(v)= \frac1{\left| \{w\colon \dist(w,v)=n\} \right|}\sum_{\dist(w,v)=n} f(w).
\end{equation}

An elementary computation shows that
%
for every $n>0$,
\begin{equation}\label{eq:recurrence_for_M1}
\mu_1\mu_n f(v)=\mu_n\mu_1f(v)= 
\frac 1{q_v +1}\; \mu_{n-1} + \frac {q_v}{q_v +1}\; \mu_{n+1}\,. 
\end{equation}
%
%
%

It follows from \eqref{eq:recurrence_for_M1} that the vector space $\mathfrak R$ is a commutative algebra generated by $\mu_1$, and on every ray $[v_0, v_1,\dotsc)$ each $\gamma$-eigenfunction of $\mu_1$ radial around $v_0$ satisfies for every $n\in\mathN$ the recurrence relation 
\begin{equation}\label{eq:semihomogeneous_recurrence}
\begin{split}
\gamma\,f(v_0)  &= f(v_1) \qquad\text{if }|v_1|=1;
\\[.2cm]
\gamma\,f(v_n)&=\frac1{q_{v_n} +1} \; f(v_{n-1}) + \frac{q_{v_n}}{q_{v_n} +1}\; f(v_{n+1}) \qquad\text{if }|v_n|=n>0.
\end{split}
\end{equation}

In particular, for each eigenvalue $\gamma$ there is exactly one $\gamma$-eigenfunction $f$ radial around any given vertex $v_0$ and satisfying the initial condition $f(v_0)=1$. If $\gamma= 0$, this eigenfunction must vanish on each $v$ with $|v|=1$ by the first identity in \eqref{eq:semihomogeneous_recurrence}, henceafter on all of $V_-$ by the second identity; it is now easy to compute this eigenfunction, that we denote by $\psiphi (v,v_0\barra 0)$, in the usual case $v_0\in V_+$:
\begin{equation}\label{eq:the_semi-homogeneous_spherical_function_of_eigenvalue_zero}
\psiphi (v,v_0\barra 0)=
\begin{cases}
(-1)^{\frac{|v|}2} q_-^{-\frac{|v|}2}&\qquad\text{if $|v|$ is even},\\[.2cm]
0                       &\qquad\text{if $|v|$ is odd}.
\end{cases}
\end{equation}
%
%
%

Let
\begin{equation}\label{eq:pcrit}
\pcrit=
\frac{\ln(q_+q_-)}{\ln q_-}.
\end{equation}
If  $C(n)=\{v\colon |v|=n\}$, it is immediately seen that, for every $n$, $\left|C(2n) \right|=(q_+q_-)^n$.
This yields the following $\ell^p$-behavior: 
\begin{equation}\label{eq:phi(v,v_0,0)_in_l^p_iff_p>p_crit}
\psiphi (\missarg,v_0\barra 0)\in \ell^p(V) \text{ if and only if } p>\pcrit.
\end{equation}
In particular,
$
\psiphi (\missarg,v_0\barra 0)\in \bigcap_{p>\pcrit}\ell^p(V)$.
Note that $\pcrit=2$ in the homogeneous setup $q_+=q_-$, and  $\pcrit>2$ if and only if $q_+>q_-$. Therefore
\begin{equation}\label{eq:phi(v,v_0,0)_in_l^2_iff_q_+<q_-}
\psiphi (\missarg,v_0\barra 0)\in \ell^2(V) 
\Longleftrightarrow q_+<q_-.
\end{equation}
%
%
%
As a consequence, if (and only if) $q_+<q_-$, then $0$ belongs not only to $\spectrum(\mu_1;\,\ell^2(V))$, but also to the discrete spectrum of $\mu_1$ as an operator on $\ell^2$ and more generally on $\ell^p$ for every $p\geqslant 2$.
\end{remark}

\begin{definition}\label{def:spherical_functions}
For $\gamma\in\mathC$, the $\gamma$-eigenfunction of $\mu_1$ radial around $v_0$ with value 1 at $v_0$ is called \emph{spherical function} with eigenvalue $\gamma$ and is denoted by $\phi(\missarg,v_0\barra \gamma)$.
\end{definition}

\section{The convolution algebra of radial functions and its multiplicative functionals}\label{Sec:multiplicative_functionals}

Given a locally compact group $G$ and a compact subgroup $K$,  the pair $(G,K)$ is a \emph{Gelfand pair} if the convolution algebra $L^1(K\backslash G/K)$ is commutative. Here we are interested in the set-up where $G=\AutT$ is the group of automorphisms of an infinite homogeneous or semi-homogeneous tree.

The isotropy subgroup $\AutT_v$ of $\AutT$ at any $v\in V$ is compact, and  $\AutT/\AutT_v$  
is discrete. 
If $T$ is homogeneous, $\AutT$ acts transitively upon  $V$, and  $\AutT/\AutT_v$ is in bijection with $V$. Moreover, $\AutT_v$ acts transitively upon the circle $C_n$ of vertices at any distance $n$ from $v$, or equivalently, the action of $G$  is doubly transitive. On the othe hand, if $T$ is semi-homogeneous but not homogeneous, $G$ has the two orbits  $V_\pm\subsetneq V$, and for any $v_+\in V_+$, $v_-\in V_-$, $n\in\mathN$, $K_{v_\pm}$ acts transitively upon the circle of vertices in $V_\pm$ at distance $n$ from $v_\pm$, respectively. Therefore $\AutT_{v_\pm}$ in in bijection with $V_\pm$.  We shall restrict attention to the semi-homogeneous not homogeneous setting, and the orbit $V_+$, with reference vertex $v_0$.

By this bijection, summable functions on the discrete space $V_+$ lift to summable functions on $G$ (with respect to its Haar measure). Hence, the convolution product on $G$ gives rise to a convolution product on  $\ell^1(V_+)$, and in particular on finitely supported functions therein; we shall often restrict attention to this space. Note that the liftings from $G/\AutT_{v_0}$ to $G$ identify right $\AutT_{v_0}$-invariant functions on $G$ with functions on $V_+$. Write $K=\AutT_{v_0}$: then  
the two-sided $K$ invariant functions on $G$ are \emph{radial} functions on  $V_+$, in the sense that they depend only on the distance from  $v_0$. We denote the corresponding $\ell^1$ space by $\ell^1_\#(V_+)$. We shall show that $(\AutT,\AutT_{v_0})$ is a Gelfand pair, that is, $\ell^1_\#(V_+)$ is an abelian convolution algebra. All this also works word by word for $V_-$. In the special case of homogeneous trees, the convolution induced by $\AutT$ was studied in \cites{
Casadio-Tarabusi&Gindikin&Picardello-Book,  Casadio-Tarabusi&Picardello-Radon_on_trees_and_counterparts}; some preliminary facts about convolution in the  semi-homogeneous settings are in \cites{CCKP, Casadio-Tarabusi&Picardello-algebras_generated_by_Laplacians},

Consider a function $f:G \mapsto \mathC$, and the Haar measure $\mu$ on $\AutT$ normalized on $K$. The usual definition of convolution for functions on $\AutT_v$ is $u_1*u_2(\tau)= \int_{\AutT_{v_0}} u_1(\lambda^{-1}\tau)\,u_2(\lambda)\,d\mu(\lambda)$, since $\AutT_{v_0}$ is unimodular; this product is associative. Let $\lambda\mapsto \widetilde\lambda$ be the canonical projection of $\AutT_{v_0}$ to $\AutT/\AutT_{v_0}$, $\widetilde\mu$ the quotient measure of $\mu$ (that is, the counting measure) on $\AutT/\AutT_{v_0}$.  
Let us assume that $u_2$ is right-$\AutT_{v_0}$-invariant and $u_1$ bi-$\AutT_{v_0}$-invariant. Then  $u_1*u_2$ is bi-$\AutT_{v_0}$-invariant. Indeed, the right-invariance is clear, and the left-invariance follows from unimodularity:  for every $\kappa\in \AutT_{v_0}$, one has $u_1*u_2(\kappa\tau)= \int_G u_1((\kappa^{-1}\lambda)^{-1}\tau)\,u_2(\lambda)\,d\mu(\lambda)=
\int_{\AutT} u_1(\lambda^{-1}\tau)\,u_2(\kappa\lambda)\,d\mu(\lambda)=u_1*u_2(\tau)$. Moreover, the convolution can be regarded as a product between $u_1$ on $\AutT_{v_0}\backslash \AutT/\AutT_{v_0}$  and  $u_2$ on $\AutT/\AutT_{v_0}$ and functions as follows:
\begin{align*}
u_1*u_2(\widetilde\tau)&=\int_{\AutT_{v_0}}\int_{\AutT_{v_0}}  u_1(\kappa^{-1}\lambda^{-1}\tau)\,u_2(\lambda\kappa)\,d\mu(\kappa)\,d\widetilde\mu(\widetilde\lambda)\\[.2cm]
&= \int_{\AutT_{v_0}} u_1\left({{\lambda^{-1}\tau}}\right)\,u_2(\lambda)\,d\widetilde\mu(\widetilde\lambda)
= \int_{\AutT_{v_0}} u_1\left({\widetilde{\lambda^{-1}\tau}}\right)\,u_2(\widetilde\lambda)\,d\widetilde\mu(\widetilde\lambda)
.
\end{align*}
Because of the bijections introduced above, this defines a convolution product between functions on $V_+$ and radial functions on $V_+$. Now assume that also $u_1$ be bi-$\AutT_{v_0}$-invariant.

If $f$ and $g$ are regarded as functions on $V_+$, with $f$ radial around $v_0$,
their convolution on  $V_+$ becomes

\begin{equation}\label{eq:homogeneous_convolution}
f*g(v)=\sum_{v'\in V} f(\dist(v,w))\,g(w).
\end{equation}
Every $\tau\in \AutT_{v_0}$ extends to an operator on functions on $V_+$ by the rule
 $\tau g(v) = g(\tau^{-1} v)$. For $V\in V_+$ choose $\tau[v]\in \AutT_{v_0}$ such that $\tau[v]\,(v_0)=v$; the choice of $\tau[v]$ is determined only up to elements in its right coset
modulo  $K$; we set $\tau[v] \delta_w(u)=\delta_w(\tau[v]u)$. Then the definition \eqref{eq:homogeneous_convolution} of convolution of $f$ and $g$ with $f$ radial is equivalent to
\begin{equation}\label{eq:convolution_as_inner_product}
f*g(v) = \langle \tau[v]\, f,\,g\rangle,
\end{equation}
where $\langle\missarg,\,\missarg\rangle$ denotes the inner product in $\ell^2(V_+)$.
Of course, this implies
\[
f*g(v_0)=\langle f,\,g\rangle.
\]
and, for $f,g,h\in\ell^1_\#(V_+)$
\begin{equation}\label{eq:associativity}
\langle f,\,g*h\rangle = f*(g*h)(v_0)=(f*g)*h(v_0)=\langle  f*g,h\rangle.
\end{equation}

Now, if $f$ and $g$ are both bi-$K$-invariant, i.e., radial around $v_0$, and $v\in V_+$,
\begin{equation}\label{eq:radial_homogeneous_convolution}
f*g(v)=f*g(\dist(v,v_0))=\sum_{w\in V_+} f(\dist(v,w))\, g(\dist(w,v_0)).
\end{equation}
So, the convolution of radial 
functions is radial, 
hence 
$\ell^1_\#(V_+)$ is a convolution algebra. This algebra is the closure in the $\ell^1$ norm of the algebra $\mathfrak{R}_\#(V_+)$ of radial finitely supported functions. 
From now on, we shall use the term \emph{radial} function on $V_+$ around $v_0$ 
instead of bi-$\AutT_{v_0}$-invariant function on $\AutT_{v_0}$. 

For all functions  $f,g$ on $V_+$ with $f$ radial, and $\tau\in \AutT_{v_0}$,
one has $\tau (f*g)=f*\tau g$, where $\tau g$ is defined by $\tau g(v) = g(\tau^{-1} v)$. 
Indeed, 
\begin{equation}\label{eq:automorphisms_commute_with_convolutions_by_radial_functions}
\begin{split}
\tau(f*g)(v) &= \sum_{v'\in V_+} g(v')\,f(\dist(\tau^{-1}v,\,v')) = \sum_{v''\in V_+} g(\tau^{-1} v'')\,f(\dist(\tau^{-1}v,\,\tau^{-1} v'') )
\\[.1cm]
&= \sum_{v''} g(\tau^{-1}v'')\,f(\dist(v,\,v'') )= f*\tau g(v).
\end{split}
\end{equation}

Denote by $\calE$  the radialization operator around $v_0$ on finitely supported functions, that is,
$\calE g (v)=\dfrac 1{\#C(|v|)} \sum_{w\in C(|v|)} g(w)$.
Then, by \eqref{eq:automorphisms_commute_with_convolutions_by_radial_functions}, for all $f,g\in\ell^1(V)$,
\begin{equation}\label{eq:radialization_commutes_with_radial_convolvers}
 \calE (f*g) = f * \calE g \quad\text{  if $f$ is radial.}
\end{equation}

Let $\mu_{n}$ be the radial function that is non-zero only if $|v|=n$ and with $\ell^1$-norm 1, that is $\mu_n=\calE \delta_v$ for $|v|=n$. Note that $\mu_n(v)=p_n(v_0,v)$, where $p_n$ are the $n$-step isotropic transition probabilities associated to the operator $M_n$ of \eqref{eq:M_n}. Therefore, by \eqref{eq:homogeneous_convolution},
\begin{equation*}
\mu_2 * \mu_{2n}
= \frac 1{(q_+ +1)q_- } \,\bigl(\mu_{2n-2} + (q_--1) \mu_{2n} + q_+q_-\mu_{2n+2}\bigr).
\end{equation*}
It follows that the radial convolution algebra $\ell^1_\#(V_+)$ is generated by  $\mu_2$, hence it is abelian.

\begin{corollary}\label{cor:Chebyshev_polynomials}
For every $n\in\mathN$, there exists a (unique) polynomial $P_{n}$ of degree $n$ such that $\mu_{2n}=P_n(\mu_2)$.
\end{corollary}

\begin{lemma}\label{lemma:properties_of_spherical_functions}
On a a semi-homogeneous tree, the following properties of a function $\phi\not\equiv 0$ on $V_+$ are equivalent:
\begin{enumerate}
\item [$(i)$]  $\phi$ on $V_+$ is 
a spherical function;
\item [$(ii)$ ] for all $v,w\in V_+$, $\calE (\tau[v]\, \phi) (w)= \phi(v)\phi(w)$;
\item [$(iii)$]
the functional $L_\phi(f)= \langle f,\,\phi\rangle$ is multiplicative on the convolution algebra $\ell^1_\#(V_+)$.
\end{enumerate}
\end{lemma}
\begin{proof} Since the radial algebra is generated by $\mu_1$, for every $n$ there exists a polynomial $Q_n$ (of degree $n$) such that $\mu_n=Q_n(\mu_1)$.
Let $\phi=\phi(\missarg,v_0\barra\gamma)|_{V_+}$, $u,v\in V_+$ and $n=|w|$. Then $\mu_n * \phi=Q_n(\mu_1) \phi =Q_n(\gamma) \phi$. By \eqref{eq:radialization_commutes_with_radial_convolvers}, $\calE(\tau[v]\, \phi) (w)=\langle  \tau[v]\, \phi, \mu_n\rangle=\phi*\mu_n(v^{-1}=Q_n(\gamma)\phi(v^{-1})=Q_n(\gamma)\phi(v)$. On the other  hand, since $\phi$ is radial, $\phi(w)=\langle \mu_n,\,\phi\rangle = \mu_n * \phi (v_0) = Q_n(\gamma)\phi(v_0)=Q_n(\gamma)$. Therefore $(i)$ implies $(ii)$.

Now let $\phi$ be a function on $V_+$ that satisfies $(ii)$ and choose $v$ such that $\phi(v)\neq 0$. If $(ii)$ holds, for every $w\in V_+$ we have $\phi(w)= \calE (\tau[v]\, \phi) (w)/\phi(v)$. Therefore $\phi$ is radial; moreover, for all radial $f,g$ on $V_+$, by \eqref{eq:convolution_as_inner_product} and \eqref{eq:radialization_commutes_with_radial_convolvers},
\begin{equation}\label{eq:inner_product_between_spherical_function_and_a_convolution}
\begin{split}
L_\phi(f*g) & =\langle f*g,\phi\rangle = 
\bigl\langle \langle \tau[\missarg]\, f,g\rangle,\, \phi(\missarg)\bigr\rangle
=
\bigl\langle f,\, \langle g, \tau[\missarg] \phi\rangle\bigr\rangle
=
\sum_v f(v) \langle \tau[v]^{-1}\phi,g\rangle 
\\[.2cm]
&= \sum_v f(v) \calE (\langle \tau[v]^{-1}\phi,g\rangle)
\sum_{v,w\in V_+} f(v)\,g(w)\,\phi(v)\,\phi(w) = L_\phi(f)\,L_\phi(g),
\end{split}
\end{equation}
since $\phi$ is radial and $|\tau[v]^{-1}(v_0)|=|\tau[v](v_0)|$.
Thus $(ii)$ implies $(iii)$.

If $(iii)$ holds, $f\in \ell^1_\#(V_+)$ and $\phi_2$ denotes the value of the radial function $\phi$ on vertices of length 2, then
\[
L_\phi(\mu_2*f) = \langle \phi,\, \mu_2\rangle\,\langle \phi,\, f\rangle =\phi_2 \,\langle \phi,\, f\rangle.
\]
On the other hand, by \eqref{eq:associativity},
\[
L_\phi(\mu_2*f)  = \langle \phi,\,\mu_2*f\rangle=\langle \phi*\mu_2,\,f\rangle
\]
Since this holds for each radial function $f$, it follows that $\phi$ is an eigenfunction of $\mu_2$ (with eigenvalue $\phi_2$). Moreover, $\phi(v_0)\neq 0$, because a radial eigenfunction of $\mu_2$ that vanish at $v_0$ must vanish everywhere. 
Now,

\begin{equation}\label{eq:phi(v_0)=1}
\phi(v_0)=\phi*\delta_{v_0}(v_0)=\phi*(\delta_{v_0}*\delta_{v_0})(v_0)=(\phi(v_0)^2,
\end{equation}
hence $\phi(v_0)= 1$, and $(i)$ follows.
\end{proof}

\begin{corollary} \label{cor:eigenfunction_properties_of_spherical_functions} If $\phi$ is a spherical function, then $\phi(v_0)=1$, $\mu_2*\phi=\gamma\phi$ with $\gamma=\phi(v)$ for $|v|=2$, and $\mu_{2n}*\phi=P_n(\gamma)\,\phi$, where $P_n$ is the polynomial of Corollary \ref{cor:Chebyshev_polynomials}..
\end{corollary}

\begin{proof} By \eqref{eq:phi(v_0)=1}, $\phi(v_0)=L_\phi(\delta_{v_0})=1$. Moreover, $\mu_2\phi=\calE(\tau[v]\,\phi)$ for any $v$ with $|v|=2$. Therefore, by Lemma \ref{lemma:properties_of_spherical_functions}\,$(ii)$, 
$\mu_2\phi=\phi(v)\,\phi$ for every such $v$. The remainder of the statement follows from Corollary \ref{cor:Chebyshev_polynomials}.
\end{proof}

\section{Positive definite spherical functions on \texorpdfstring{$V_+$}{V+}}\label{Sec:positive_definite}
\begin{definition} \label{def:positive-definite_functions} $\ell^1(V_+)$ and $\ell^1_\#(V_+)$ are involutive algebras equipped with the involution
$f^*(v)=f^*(\tau[v](v_0))=\overline{f(\tau[v]^{-1}(v_0)}$.

A spherical function is positive definite for $\AutT$ acting on $V_+$ if it induces a positive functional on the involutive algebra $\ell^1(V_+)$.

It is well known that a right-$\AutT_{v_0}$-invariant function  $\phi$ on $\AutT$, or equivalently a  function on $V_+$, is positive definite (with respect to $\AutT$) if and only if it is a matrix coefficient of a unitary representation $\pi_\phi$ of $\AutT$ (acting on a vector space equipped with an inner product) that with respect to this inner product is of the type 
\begin{equation}\label{eq:positive_definite_functions_as_entries_of_unitary_representations}
\phi(v)=\langle h, \pi_\phi(\tau[v]) h\rangle.
\end{equation}
 Here we shall regard $h$ as a function on $V_+$, the action of $\tau=\tau[v]\in\AutT$ being right-$\AutT_{v_0}$-invariant.

Let $\phi$ be a positive definite function on $V_+$ and $\calV_\phi$ the linear span of all translates of $\phi$ under the action of $\Aut(T)$.
 Then $\phi$
 induces a 
            positive semi-definite 
 inner product on $\calV_\phi$ by the rule
 \begin{equation}\label{eq:the_phi-inner_product}
 \langle \tau[v] \phi,\,\tau[w]\phi \rangle_\phi = \phi(\tau[v]^{-1} w).
 \end{equation}
 From now on, we shall denote by $\langle\missarg,\,\missarg\rangle_\phi$ this inner product and with $\langle\missarg,\,\missarg\rangle$ the $\ell^2$-inner product.
 
 Denote by $\calN_\phi$ the subspace of $\calV_\phi$ of all function $f$ such that $\langle f,\,f\rangle_\phi=0$, and $\calH_\phi=\calV_\phi/\calN_\phi$. Then $\Vert f \Vert_\phi=\sqrt{\langle f,\,f\rangle_\phi}$ is  a Hilbert space norm on $\calH_\phi$ (the so-called \emph{GNS-norm}).
 
 Let now $\phi=\phi(\missarg,v_0\barra \gamma)$ and write
 $\calV_\gamma$ instead of $\calV_\phi$, that is the linear span of the functions $\{v\mapsto \phi(\tau v,v_0\barra \gamma)\colon \tau\in\AutT\}$. Then 
 $\calV_\gamma$ is invariant under $\AutT$, and 
 the action of $\AutT$ gives rise to an (algebraic) representation $\pi_\gamma$ of $\AutT$ on $\calV_\gamma$. If $\phi$ is positive definite, then $\pi_\gamma$ extends to a topological representation of $\AutT$ on  the Hilbert space closure $\overline\calH_\gamma$ of$\calH_\gamma= \calV_\gamma / \calZ_\gamma$, called the spherical representation at the eigenvalue $\gamma$.
\end{definition}

\begin{corollary}\label{cor:positive_definite_spherical_functions}
A spherical function 
defines a positive functional on the involutive algebra $\ell^1(V_+)$ if and only if it is bounded and real-valued.
\end{corollary}
\begin{proof}
Let $\phi$ be a bounded spherical function and $f\in\ell^1_\#(V_+)$. Then $\phi$ is radial, and $\phi(\tau[v]^{-1}(v_0))=\phi(v)$ because 
\begin{equation}\label{eq:inversion_in_Aut(V_+)_preserves_the_length}
|(\tau[v]^{-1}(v_0)|=|\tau[v](v_0)|=|v|.
\end{equation}
By Lemma \ref{lemma:properties_of_spherical_functions}, or more directly by \eqref{eq:inner_product_between_spherical_function_and_a_convolution},   if $\phi$ is real, 
\[
\begin{split}
L_\phi(f^* *f)&=
L_\phi(f^*)\,L_\phi(f) 
= \Bigl(\sum_{v\in V_+}\overline f(\tau[v]^{-1}(v_0)\,\phi(v)\Bigr)\Bigl(\sum_{w\in V_+}  f(w) \phi(w)\Bigr)
\\[.2cm]
&
= \Bigl(\sum_{v\in V_+}\overline f(v)\,\phi(\tau[v](v_0))\Bigr)\Bigl(\sum_{w\in V_+}  f(w) \phi(w)\Bigr)
\\[.2cm]
&
= \Bigl(\sum_{w\in V_+}  \overline f(v)\, \phi(v)\Bigr) \Bigl(\sum_{w\in V_+}  f(w) \phi(w)\Bigr)
\\[.2cm]
&
=\langle \overline f,\,\phi\rangle\,\langle f,\,\phi\rangle=
|\langle f,\,\phi\rangle|^2 \geqslant 0,
\end{split}
\]
%
hence $L_\phi$ is a positive functional on the involutive algebra $\ell^1_\#(V_+)$. Now let $h\in\ell^1(V_+)$. By \eqref{eq:radialization_commutes_with_radial_convolvers}, $L_\phi(h)=\langle h,\,\phi\rangle = \phi*h(v_0) =\phi*\calE h (v_0)=
\langle \calE h,\,\phi\rangle = L_\phi(\calE h) $. Therefore, by \cite{Figa-Talamanca&Picardello}*{Chapter 3, Lemma 1.2 and 1.3}, $L_\phi$ is also a positive functional on the involutive algebra $\ell^1(V_+)$. Thus $\phi$ is positive definite.

Conversely, let $\phi$ be a positive definite function on $V_+$. Then $\phi$ is bounded (by its value at $v_0$), and $\phi(\tau[v]^{-1} v_0) =\overline\phi (\tau[v] v_0)=\overline\phi(v_0)$ for every $v\in V_+$.
Since $\phi$ is radial, \eqref{eq:inversion_in_Aut(V_+)_preserves_the_length} implies that it is real-valued.
\end{proof}

 %
 \begin{corollary}
 If a spherical function $\phi$  on $V_+$ belongs to $\ell^2(V_+)$, then $\phi$ is positive definite, 
 and every function $f$ on $V_+$,   its $\ell^2$-norm and its norm $\| \missarg \|_\phi$ defined by the inner product \eqref{eq:the_phi-inner_product} are related by $\| f \|_{\ell^2(V_+)}=\| \phi \|^2_{\ell^2(V_+)} \, \Vert f \Vert_\phi$.
 \end{corollary}
 
 \begin{proof}
 Let $\pi$ be the right regular representation of $\AutT$ on functions on $\AutT/\AutT_{v_0}\equiv V_+$, that is, $\pi(\tau[v])h(w)=h(\tau[v]^{-1}w)$ for every function $h$ on $V_+$ and every $v,w\in V_+$;  clearly, this is a unitary representation.
  Then, by Lemma \ref{lemma:properties_of_spherical_functions}\,$(ii)$,
 %
\begin{equation}\label{eq:matrix_coefficients_of_the_reguar_representation_in_the_ell^2-norm}
  \langle \pi(\tau[v])\phi ,\,  \phi \rangle =  \langle \calE (\pi(\tau[v])\phi),\, \phi \rangle = \phi(v)\,\| \phi\|^2_{\ell^2(V_+)}.
 \end{equation}
 Therefore the function $\phi$ is a (positive) multiple of a matrix coefficient of a unitary representation, hence it is positive definite by Definition \ref{def:positive-definite_functions}.
 %
On the other hand, by \eqref{eq:the_phi-inner_product}
\begin{equation}\label{eq:matrix_coefficients_of_the_reguar_representation_in_the_phi-norm}
  \langle \pi(\tau[v])\phi,\,  \phi \rangle_\phi =  \phi(v).
 \end{equation}
 The statement follows by comparing \eqref{eq:matrix_coefficients_of_the_reguar_representation_in_the_ell^2-norm}
 and \eqref{eq:matrix_coefficients_of_the_reguar_representation_in_the_phi-norm}.
 \end{proof}

  \begin{remark}
 If $\phi=\phi(\cdot,v_0\barra\gamma)\in\ell^2(V_+)$ and $\gamma\in\spectrum_{\ell^1(V_+)} \mu_2$, then we have $\Vert f \Vert_\phi=\| \phi \|^2_{\ell^2(V_+)}\| f \|_{\ell^2(V_+)}$.

Indeed, by Corollary \ref{cor:positive_definite_spherical_functions}, a spherical function on $V_+$  is positive definite if and only if it is real-valued and bounded. On the other hand,
it is known \cite{Casadio-Tarabusi&Picardello-semihomogeneous_spectra} that a spherical function on $V_+$ is real valued and bounded if (and only) if its eigenvalue belongs to $\spectrum_{\ell^1(V_+)} \mu_2$.
 \end{remark}

\section{Spherical representations of \texorpdfstring{$\AutT$}{Aut T}}\label{Sec:spherical_representations}
On the basis of the previous Sections, the results of  \cite{Figa-Talamanca&Picardello} now open
the way to the spherical representation theory of $\AutT$:

 \begin{theorem}\label{theo:square-integrable_representation,unitarizable_representations,Irreducible_representations_of_AutT}
 Denote by  $S_1$ the spectrum of $\mu_1$ on $\ell^1(V)$, and by $S_2$ its spectrum on $\ell^2(V)$.
 \begin{enumerate}
 \item [$(i)$]
$\phi(\tau v,v_0\barra \gamma)$ is real-valued and bounded if and only if $\gamma\in S_1 \cap \mathR$. For these eigenvalues, the representation $\pi_\gamma$ is a unitary representation on the Hilbert space $\overline\calH_\gamma$.

\item[$(ii)$]
 The Plancherel measure associated to the regular representation
 $\rho$ of $\AutT$ on $\ell^2(V)$ is absolutely continuous with respect to Lebesgue measure on $ S_2\subsetneq\mathR$, except for an atom at the eigenvalue $0$ if (and only if) $q_+<q_-$. For these eigenvalues $\gamma$, the representation $\pi_\gamma$ is weakly contained in the regular representation $\rho$.

 \item[$(iii)$] If $\gamma=0$ and $q_+<q_-$, then $\phi(\tau v,v_0\barra 0)\in \ell^2(V)$ and, on the norm-closure of $\calV_0$, $\pi_0$ is a square-integrable representation, that is, a subrepresentation of the regular representation of $\AutT$ on $\ell^2(V_+)$.

 \item [$(iv)$] All the representation $\pi_\gamma$ on $\ell^2(V_+)$ for $\gamma\in \Int S_1$ are irreducible.

 \item[$(v)$] The representations $\rho$ and $\pi_\gamma$ can be defined also on spaces of functions on $V_-$ and on $V$. Since the action of $\AutT$ preserves $V_\pm$, $\rho$ and each $\pi_\gamma$, when regarded on $V)$, decomposes as direct sum of its restrictions to $V_+$ and $V_-$. In particular, for $\gamma\in \Int S_1$,\,$\pi_\gamma$ is the direct sum of two irreducible subrepresentations.
 \end{enumerate}
 \end{theorem}
 
 \begin{proof}
For part $(i)$ it is enough to show that $\phi(\tau v,v_0\barra \gamma)$ is real-valued and bounded if and only if $\gamma\in S_1$, and that the norm  $ \Vert \missarg \Vert_\phi$ is the $\ell^2$-norm if and only if $\gamma\in S_2$: this
follows from \cite{Casadio-Tarabusi&Picardello-semihomogeneous_spectra}*{Theorem 5.8}.

For part $(ii)$, it has been observed in \cite{Casadio-Tarabusi&Picardello-semihomogeneous_spectra} (see also \cite{Casadio-Tarabusi&Picardello-algebras_generated_by_Laplacians}) that
the set $V_+$, when regarded as the Cayley graph of the step-2 isotropic operator $\mu_2$, becomes a polygonal graph in the sense of \cite{Iozzi&Picardello-Springer}, therefore the regular representation of $\AutT$ on $\ell^2(V_+)$ can be realized by the action on this polygonal graph. In this setting, the Plancherel measure $m_\rho$ was computed in \cite{Faraut&Picardello}. An accurate comparison of our current setting with that of \cites{Iozzi&Picardello-Springer, Faraut&Picardello} shows that $m_\rho$ is
 absolutely continuous with respect to Lebesgue measure, with a continuous Radon-Nykodim derivative, except for an atom at $\gamma=0$ if $q_+<q_-$, and $\rho$ decomposes on $\ell^2(V_+)$ as $\rho=\int_{S_2}^\oplus \pi_\gamma \,dm(\gamma)$. For each $\gamma\in S_2$ and for $\delta> 0$ the matrix coefficients of $\pi_\gamma$ is  a uniform limit  on finite sets, as $\delta\to 0$, of coefficients of the representation
$\frac1\delta \int_{U_\delta}^\oplus \pi_\gamma \,dm(\gamma)$ where $U_\delta$ is a neighborhhod of $\gamma$ of radius $\delta$. Therefore $\pi_\gamma$ is weakly contained in $\rho$.

Part $(iii)$ follows  from \eqref{eq:phi(v,v_0,0)_in_l^2_iff_q_+<q_-}. 

Part $(iv)$ is proved as in the setting of homogeneous trees \cites{Figa-Talamanca&Picardello-JFA,Figa-Talamanca&Picardello}; see also \cites{Figa-Talamanca&Steger,Figa-Talamanca&Nebbia}.

Part $(v)$ is clear.
\end{proof}

%
%
%


\end{document}